\documentclass[11pt]{amsart}
  
\input xy
\xyoption{all}
  
\usepackage{amssymb}
\usepackage{amsmath}
\usepackage{amscd}
\usepackage{latexsym}
\usepackage{bbold}
\begin{document}
 
\newtheorem{lemma}{Lemma}[section]
\newtheorem{prop}[lemma]{Proposition}
\newtheorem{cor}[lemma]{Corollary}
  
\newtheorem{thm}[lemma]{Theorem}
\newtheorem{Mthm}[lemma]{Main Theorem}
\newtheorem{con}{Conjecture}
\newtheorem{claim}{Claim}
\newtheorem{ques}{Question}

\theoremstyle{definition}
  
\newtheorem{rem}[lemma]{Remark}
\newtheorem{rems}[lemma]{Remarks}
\newtheorem{defi}[lemma]{Definition}
\newtheorem{ex}[lemma]{Example}
                                                                                
\newcommand{\C}{{\bf C}}
\newcommand{\R}{{\bf R}}
\newcommand{\Q}{{\bf Q}}
\newcommand{\Z}{{\bf Z}}
\newcommand{\N}{{\bf N}}
\newcommand{\LS}{{\mathcal L}(S^1)}
\newcommand{\LDS}{L^2(S^1)}
\newcommand{\CS}{{\mathcal C}(S^1)}
\newcommand{\DS}{{\mathcal C}^\infty(S^1)}
\newcommand{\ch}{{\mathcal F}}

\newcommand{\cg}{{\bf C}G}
\newcommand{\eog}{\ell^1(G)}
\newcommand{\etg}{\ell^2(G)}
\newcommand{\epg}{\ell^p(G)}
\newcommand{\eqg}{\ell^q(G)}
\newcommand{\vNG}{{\mathcal N}\! G}
\newcommand{\ZNG}{Z({\mathcal N}\! G)}
\newcommand{\ag}{{\mathcal A}G}
\newcommand{\crg}{C^*_rG}
\newcommand{\cmg}{C^*G}
\newcommand{\cH}{\mathcal H}
\newcommand{\BH}{{\mathcal B}({\mathcal H})}
\newcommand{\UH}{{\mathcal U}({\mathcal H})}
\newcommand{\BG}{{\mathcal B}({\etg})}
\newcommand{\UG}{{\mathcal U}({\etg})}

\title{Introduction to the Rapid Decay Property}
\author{Indira Chatterji}
\thanks{{The author is partially supported by the Institut Universitaire de France (IUF)}}
\begin{abstract}This is an introduction to the Rapid Decay property, with a survey of known results and equivalent definitions of this property. We also discuss in details the easy case when $G=\Z$. Everything in this paper is well-known by different sets of people.\end{abstract}

\maketitle

In this paper, $G$ will be a countable group, that we assume finitely generated by a finite generating set $S$. The \emph{word length} of an element $\gamma\in G$ is the minimal number of elements of $S$ needed to express $\gamma$. The group $G$ acts on $\etg$ via the (left) regular representation, and one extends this action to $\cg$ and the bounded operators on $\ell^2(G)$ (see Section \ref{basics} for the definitions). 
\begin{defi}\label{RD1erjet}A finitely generated group $G$ has \emph{the Rapid Decay property} (we shall say RD for short) if there are constants $C$ and $D$ such that for any $R\in\N$ and any $f\in\cg$ such that $f$ is supported on elements shorter than $R$, the following inequality holds:
$$\|f\|_*\leq CR^D\|f\|_2,\eqno{RD(1)}\label{RD1}$$
where $\|\ \|_*$ is the operator norm of elements in $\cg$ acting on $\ell^2(G)$ via the linearization of the left regular representation and $\|\ \|_2$ is the $\ell^2$-norm on $\cg$ (see Section~\ref{basics}). 
\end{defi}
Because of the constant $C$, the above inequality is interesting only for large $R$'s. One can notice that whereas the left hand side of the above inequality depends on the group operation (via the regular representation) the right hand side doesn't and only depends on the geometry of the group in a ball of radius $R$. Depending on how the group operation allows to combine $\gamma$'s and $\mu$'s in $G$ to obtain elements $\gamma\mu^{-1}$ of length shorter than $R$, one may or may not be able to control the left hand side. Groups with the Rapid Decay property are exactly the ones for which we can control the left hand side. Several other equivalent definitions of with the Rapid Decay property will be discussed in Section~\ref{equivalences}.

\smallskip

The terminology \emph{Rapid Decay} comes from Connes' original definition, equivalent to RD(1) above, which is as follows (see Definitions \ref{reducedCstar} and \ref{rdfunctions} for all the definitions).
\begin{defi}A group has the Rapid Decay property if 
$$H^\infty_\ell(G)\subseteq C^*_r(G)\eqno{RD(2)}\label{RD2}$$
that is, the set of rapidly decaying functions on $G$ is a subalgebra of the reduced C*-algebra of $G$. 
\end{defi}
Indeed, it is a classical result that any abelian and unital C*-algebra is the algebra of continuous functions on a a compact space (its dual), so philosophically, a non-abelian unital C*-algebra (such as, for instance, the reduced C*-algebra of a finitely generated group), should correspond to the continuous functions on a non-commutative space. So, for a finitely generated group $G$, admitting that the reduced C*-algebra of $G$ corresponds to continuous functions, it is natural to look for the smooth functions. In Section~\ref{Z} we shall throughly see the case of the group $\Z$. Even though this case is technically very easy it illustrates nicely Connes' definition.

\medskip

Amenable groups is an important class of groups that give us the simplest examples of non RD groups, that is amenable groups with super-polynomial growth. An example is the following semi-direct product:%
$$\Z^2\rtimes_{\left(\begin{array}{cc}2&1\\ 1&1\end{array}\right)}\Z\simeq\{\left(\begin{array}{cc}1& \!\!0\ \ \,0\\ \begin{array}{c}a\\b\end{array}&\left(\begin{array}{cc}2&1\\ 1&1\end{array}\right)^n\end{array}\right)|\ (a,b)\in\Z^2,\  n\in\Z\}<SL_3(\Z).$$
From the definition of Rapid Decay property it is clear that the property is inherited by subgroups with the induced length, and hence having an amenable subgroup of  super-polynomial growth is an obstruction to the Rapide Decay property, showing that for instance $SL_3(\Z)$ does not have that property. The above example also shows that in general the Rapid Decay property is not stable under taking extensions, however, in a short exact sequence of finitely generated groups
$$\{e\}\to Z\to G\to Q\to\{e\}$$
then $G$ has the Rapid Decay property if and only if $Q$ has the Rapid Decay property and $Z$ has the Rapid Decay property for the induced length from $G$, see \cite{Ga} for the general statement, and Remark \ref{otherlength} for a discussion on the Rapid Decay property and length functions.
\goodbreak 
\section*{Picture of the situation} 
Here is a non-exhaustive and redundant list of what is known regarding the property of Rapid Decay:
\subsection*{Examples of discrete groups with the Rapid Decay property}
\begin{itemize}
\item Polynomial growth \cite{Jolissaint} 
\item Free groups \cite{Haagerup} 
\item Hyperbolic groups \cite{Jolissaint}, \cite{Harpe} 
\item Coxeter groups \cite{wKim}     
\item Discrete cocompact $\Gamma$ in 
\begin{itemize}\item  $SL_3(\Q_p)$ \cite{RRS} 
\item $SL_3(\R)$ and $SL_3(\C)$ \cite{Lafforgue}  
\item $SL_3(\mathbb{H})$ and $E_{6(-26)}$ \cite{moi} 
\item  products of the above \cite{moi} 
\end{itemize}
\item All rank one lattices  \cite{wKim} and more generally hyperbolic group relatively to subgroups with the Rapid Decay property \cite{DS}.
\item Cocompact  cubical  CAT(0)  \cite{wKim}
\item Mapping class groups \cite{BM}
\item Braid groups \cite {BP} and \cite{BM}
\item Large type Artin groups \cite{CHRArtin}
\item 3-manifold groups not containing $Sol$ \cite{Gau}
\item Wise non-Hopfian group \cite{BPW}
\item Small cancellation groups \cite{AD} and \cite{Osin}.
\end{itemize}
\subsection*{Examples of discrete groups without the Rapid Decay property}
\begin{itemize}
\item Amenable groups with exponential growth \cite{Jolissaint}  
\item $BS(n,m)$
\item $\Z^{\infty}$, Thompson's group
\item $SL_n(\Z), n\geq 3$ and  $Sp_n(\Z), n\geq 4$ and more generally non-uniform lattices in higher rank
\item $GL_2({\mathbb{F}}_p[t,t^{-1}])$ \cite{wKim}
\item Intermediate growth groups  \cite{Jolissaint} 
\item $\Z*\Z^2*\Z^3*\cdots$ (see \cite{Sa} for a length on that group)
\end{itemize}
\subsection*{Examples of groups for which the Rapid Decay property is open}
\begin{itemize}
\item Out$({\bf F}_n)$, $n\geq 3$
\item Cocompact lattices in $SL_n(\Q_p)$ or $SL_n(\R)$, $n\geq 4$ and more generally semisimple Lie groups, see Conjecture \ref{Valette} below
\item Artin groups
\item Cocompact CAT(0) groups
\end{itemize}
According to Jolissaint in \cite{Jolissaint}, the Rapid Decay property is preserved by free products and some amalgamated products, as well as some central extensions and semi-direct products. Moreover according to Ciobanu, Holt and Rees in \cite{CHR}, the Rapid Decay property is preserved by graph products, hence one gets more examples combining the above. However it in unknown if the Rapid Decay property is preserved under quasi-isometries, although the methods used so far to establish the Rapid Decay property are.

\medbreak

Actually the situation is much simpler than what the above list expresses: among finitely presented groups, the only known obstruction to the Rapid Decay property is to contain an amenable subgroup of super-polynomial growth for the induced length, phenomena easily observed in the presence for instance of an exponentially distorted copy of $\Z$. In the case of finitely generated groups however, Sapir in \cite{Sa} constructs a 2-generator groups without the Rapid Decay property and without super-polynomial growth amenable subgroups.
\subsection*{Acknowledgements}I would like to thank Mark Sapir and Christophe Pittet for comments on a preliminary version, as well as Adrien Boyer for discussions on the Rapid Decay property characterization and reminding me of \cite{Perrone}.
\section{Short historical survey and applications}
 First established for free groups by Haagerup in \cite{Haagerup}, the Rapid Decay property has been introduced and studied as such by Jolissaint in \cite{Jolissaint}, who notably established it for groups of polynomial growth, and for classical hyperbolic groups. The extension to Gromov hyperbolic groups is due to de la Harpe in \cite{Harpe}. One of the earliest definitions was given by Connes in \cite{ConNCG} in a non-commutative geometry setting, and the first important application of the Rapid Decay property is in Connes and Moscovici's work \cite{CoMo}, proving the Novikov conjecture for Gromov hyperbolic groups. Providing the first examples of higher rank groups, Ramagge, Robertson and Steger in \cite{RRS} proved that the Rapid Decay property holds for $\tilde{A}_2$ and $\tilde{A}_1\times\tilde{A}_1$ groups, and Lafforgue, using a very nice quasification of their proof, established it for cocompact lattices in $SL_3({\bf R})$ and $SL_3({\bf C})$ in \cite{Lafforgue}. Lafforgue proved this property as part of his big results on the Baum-Connes conjecture, that we discuss a bit more below. His result on the Rapid Decay property was generalized in \cite{moi} to cocompact lattices in $SL_3({\bf H})$ and $E_{6(-26)}$ as well as in a finite product of rank one Lie groups. We saw that $SL_n({\bf Z})$ for $n\geq 3$ (and more generally any non-cocompact lattice in higher rank simple Lie groups as they contain exponentially distorted copies of $\Z$, see \cite{LMR}) does not have the Rapid Decay property, and Valette conjectured the following.
 \begin{con}[Valette \cite{Valette}]\label{Valette}Cocompact lattices in a semisimple Lie group (real or $p$-adic) have the Rapid Decay property.\end{con}
In rank one the situation is different as we show with Ruane in \cite{wKim}; there all lattices have the Rapid Decay property. This uses that those lattices are hyperbolic relative to polynomial growth groups. More generally, Drutu and Sapir \cite{DruSa} show that any hyperbolic group relatively to subgroups with the Rapid Decay property, has the Rapid Decay property as well. More groups have recently been added to the list of groups with the Rapid Decay property, including mapping class groups by Behrstock and Minsky in \cite{BM} (those include braid groups \cite {BP} and \cite{BM}), large type Artin groups by Ciobanu Hold and Rees in \cite{CHRArtin}, as well as Wise non-Hopfian group by Barr\'e and Pichot in \cite{BPW} or coarse median groups, see Bowditch \cite{Bow}. Some small cancellation groups satisfy the Rapid Decay property, according to  Arjantseva and Dru\c tu \cite{AD} or Osin \cite{Osin}.

\subsection*{Methods for proving the Rapid Decay property} Unfortunately, all the methods used so far to establish the Rapid Decay property are very similar and give a bad exponent. Those methods rely on noticing that the left hand side of RD(1) is in fact a computation on triples of points (this is easier to see by looking at RD(5) in Proposition \ref{RD}), and then reducing the computation to triples of points that are manageable. This is the reason why the Rapid Decay property is still open for lattices in $Sp_{4}({\mathbb Q}_p)$ for instance (or more generally groups of Conjecture \ref{Valette}): the reduction from \cite{RRS} or \cite{Lafforgue}  gives triples on which the computation cannot be carried on, see Talbi's work \cite{Talbi} for more on the situation on $p$-adic lattices. Sapir's recent survey \cite{Sa} gives a nice account on the similarities and differences of those methods.

\subsection*{Locally compact compactly generated groups} An approach to prove Conjecture \cite{Valette} that so far gave only false hopes is to compare the lattice and the ambient group. Indeed, all definitions of the Rapid Decay property extend quite easily to locally compact groups (see \cite{Jolissaint}), but in this settings the Rapid Decay property is inherited by {\it open} subgroups only. However, the following result is fairly elementary.
\begin{thm}[Jolissaint \cite{Jolissaint}]Let $G$ be a locally compact group and let $\Gamma< G$ be a discrete cocompact subgroup in $G$, and let $\ell$ be a continuous length function on $G$. If $\Gamma$ has the Rapid Decay property with respect to $\ell$ (restricted to $\Gamma$), then so does $G$ (with respect to $\ell$).\end{thm}
As an approach to Conjecture~\ref{Valette} one could hope that the converse of the above theorem holds true, since all semisimple Lie groups have the Rapid Decay property (according to \cite{CPSC}, it is mainly a reformulation of Kunze-Stein phenomena \cite{Hertz}, see also \cite{Boyer2} for a simpler proof). However such a guess seems mainly supported by a lack of counter-examples.
\subsection*{Radial Rapid Decay} Considering \emph{radial} functions only has been done by Valette in \cite{ValRad} to establish the Raid Decay property for radial functions for groups acting on buildings.  Here we recall that given a length function on a group $G$, the set of radial functions (in $\cg$ or even $\etg$) is the set of functions whose values are constant on spheres, that is whose values only depend on the length of an element. This result of Valette has been extended to cocompact lattices in semisimple Lie groups for the length induced from the Riemannian metric by Perrone in \cite{Perrone2}. However, the set of radial functions seems too small to get the full Rapid Decay property for $G$. For instance, in the case of locally compact non-unimodular groups, those cannot have the Rapid Decay property according to \cite{JiSch}, but those do have a version of Rapid Decay property on radial functions, see \cite{CPSC}. However, in the case of discrete groups, it would be interesting to have examples of discrete groups that do satisfy the Rapid Decay inequality on radial functions but that do not have the Rapid Decay property.
\subsection*{The Baum-Connes conjecture}The Rapid Decay property gained attention in 2001 after Lafforgue's work on the Baum-Connes conjecture showed its importance in this context and indeed, for a large class of groups like cocompact lattices in semisimple Lie groups, the Rapid Decay property would be the last step in establishing the Baum-Connes conjecture. Lafforgue's work was the first instance where it was shown that groups with property (T) satisfied the conjecture. We will not attempt to explain the Baum-Connes conjecture here and refer the interested reader to Valette's lecture notes \cite{Valette}. We will just mention what it says, namely that, for a discrete group $G$, the assembly map
$$\mu_i:K^G_i(\underline{E}G)\to K_i(C^*_r(G))\eqno{(BC)}$$
should be an isomorphism for $i=0,1$, where the left hand side is the equivariant $K$-homology of $\underline{E}G$, the classifying space for proper actions for $G$, and is a group that algebraic topologists understand and can compute in many cases. The right hand side is the topological $K$-theory of $\crg$, the reduced C*-algebra of $G$, and this is a group that is interesting for analysts but that remains quite mysterious. Lafforgue in \cite{LaffBC} showed that, for a very large class of groups (containing for instance all lattices in Lie groups, amenable groups, a-T-menable groups, hyperbolic groups, CAT(0) groups) a modified version of the Baum-Connes conjecture holds true, where the reduced C*-algebra is replaced by any \emph{unconditional completion}, that is a Banach algebra completion of the group ring that only depends on the absolute values of elements. The Rapid Decay property is providing such a completion that in addition is known to have the same K-theory as the one of the reduced C*-algebra (see Proposition \ref{HsBanach}). Notice that despite having the Rapid Decay property, the Baum-Connes conjecture is still open for mapping class groups because nobody has yet proved that they belong to Lafforgue's class (or that they are \emph{strongly bolic}).
\subsection*{Random walks}A symmetric random walk on $G$ is given by a finitely supported symmetric function $f$ with support generating $G$ and such that $\sum_{\gamma\in G}f(\gamma)=1$. For such a random walk, the return probability of the random walk after $2n$ steps is given by $f^{(2n)}(e)$, the $2n$ convolution of $f$ with itself. The asymptotics of the return probability are very interesting, however as soon as the group is non-amenable (which according to Kesten amounts to the spectral radius of $f$ being strictly smaller than 1), the decay is exponential. A finer information is hence given by
$$\varphi_f:{\bf N}\to{\bf R}, n\mapsto \rho_f^{-2n}f^{(2n)}(e)\eqno(RW)$$
where $\rho_f$ denotes the spectral radius of $f$, that is the operator norm of $f$ acting on $\etg$. This asymptotics in (RW) can depend on the function and not only of the group, as shown by Cartwright's example \cite{Car}, see also \cite{Woess}. This asymptotics has been computed by \cite{Lal} in the case of free groups, by \cite{CaGl} in the case of free products of $\Z^n$'s and recently by Gou\"ezel in the case of hyperbolic groups \cite{Gouezel}. In the cases where $G$ has the Rapid Decay propery, we can get bounds for this asymptotics, as follows.
\begin{defi}For a group $G$ with the Rapid Decay property, the \emph{Rapid Decay exponent} of $G$ is defined as
$$\alpha=\inf\{s\in{\bf R}\hbox{ such that }H^s_{\ell}(G)\subseteq C^*_r(G)\}.$$\end{defi}
The constant $\alpha$ depends on the group $G$ only and not on the choice of a compact generating set giving the length function.  We recall a very simple remark from \cite{CPSC}, Section 7.
\begin{thm}[Chatterji, Pittet, Saloff-Coste]\label{CPS}Let $G$ be a group with the Rapid Decay property, with $D$ as its Rapid Decay exponent. Then, for every $f$ a finitely supported symmetric probability measure on $G$, there exists $c_f>0$ such that for any $n\in{\bf N}$ we have
$$\varphi_f(n)\geq c_f n^{-2D},$$
for any $s\geq\alpha$.\end{thm}
\begin{proof}Since $f$ is symmetric, $f^{(2n)}(e)=\|f^{(n)}\|^2_2$ and $\rho_f^{2n}=\|f^{(n)}\|^2_*$. Moreover, $f^{(n)}$ is supported on a ball of radius $dn$ where $d$ is the diameter of the support of $f$. Hence, RD(1) gives us that
$$\rho_f^{2n}\leq Cd^{2D}n^{2D}f^{(2n)}(e)$$
and we conclude setting $c_f=C^{-1}d^{-2D}$.
\end{proof}
In case where $\rho_f<1$ (which is exactly when $G$ is non-amenable), then there is a constant $c_f>0$ so that $\varphi_f(n)\leq c_fn^{-1}$ (see \cite{CPSC2}). Hence it is very interesting to determine the best possible Rapid Decay constant of a group $G$, as it will give a bound on the possible asymptotic behaviors of random walks on $G$. However, in the case of hyperbolic groups, Gou\"ezel in \cite{Gouezel} shows that the exact asymptotic is 3/2, whereas the Rapid Decay property will only give a bound between 1 and 3, but that could also be due to the techniques used to establish the Rapid Decay property.
\subsection*{Other applications} Other applications of the Rapid Decay property include 
\begin{itemize}
\item Connes and Moscovici's work proving the Novikov conjecture for Gromov hyperbolic groups use the Rapid Decay property for hyperbolic groups \cite{CoMo}.
\item Brodzki and Niblo's result that groups that satisfy the property of Rapid Decay with respect to a conditionally negative length function have the metric approximation property \cite{BroNi}. 
\item De la Harpe, Robertson and Valette in \cite{HRV} established that in a group $G$ with the Rapid Decay property, then $\rho(\mathbb{1}_S)=|S|^{1/2}$ if and only if $S$ generates a free semi-group, where $\rho(\mathbb{1}_S)$ is the spectral radius of the caracteristic function of a finite subset $S$ of $G$.
\item he Rapid Decay property for hyperbolic groups has been used by Grigorchuck and Nagnibeda in \cite{GriNa} to compute the convergence radius of the complete growth serie of a hyperbolic group.
\item Nagnibeda and Pak in \cite{NP} show that a finitely generated non-amenable group with the Rapid Decay property has infinitesimally small spectral radius.
\item Further use of the Rapid Decay property can be found in the work of Nevo \cite{Ne1} and \cite{Ne2}, or Antonescu and Christensen in \cite{AC}.
\end{itemize}
\section{Basic definitions}\label{basics}
\begin{defi}
A \emph{length function} on a discrete group $G$ is a function $\ell:G\to{\R}_+$ satisfying: 
\begin{itemize}
\item[(1)] $\ell(\gamma)=\ell(\gamma^{-1})$ for any $\gamma\in G$.
\item[(2)] $\ell(\gamma\mu)\leq \ell(\gamma)+\ell(\mu)$ for any $\gamma,\mu\in G$.
\item[(3)] $\ell(e)=0$, where $e=1$ denotes the identity element in $G$.\end{itemize}
The constant map $G\to\R_+,\gamma\mapsto 0$ is the \emph{trivial length}. If $G$ is generated by some finite subset $S$, then the \emph{algebraic word length} $\ell_S:G\to\N$ is a length function on $G$, where, for $\gamma\in G$, $\ell_S(\gamma)$ is the minimal length of $\gamma$ as a word on the alphabet $S\cup S^{-1}$, that is,
$$\ell_S(\gamma)=\min\{n\in{\N}|\gamma=s_1\dots s_n,\,s_i\in S\cup S^{-1}\}.$$\end{defi}
Length functions and metric spaces are basically the same thing and given a metric space $(X,d)$ on which a group $G$ acts by isometries (ie $d(\gamma\cdot x,\gamma\cdot y)=d(x,y)$ for any $x,y\in X$), then 
$$\ell(\gamma):=d(x_0,\gamma\cdot x_0)$$
is a length function, for any fixed $x_0\in X$. Conversely, given a length function $\ell$ on $G$, if $N=\{\gamma\in G|\ell(\gamma)=0\}$, we set $X=G/N$ and $d(\gamma N,\mu N)=\ell(\mu^{-1}\gamma)$. Then one can check that $(X,d)$ is a well-defined metric space on which $G$ acts by isometries.

A length function $\ell'$ \emph{dominates} another length function $\ell$ if there are two constants $A,B\geq 1$ such that
$$\ell(\gamma)\leq A\ell'(\gamma)+B,$$
for any $\gamma\in G$. We say that $\ell$ and $\ell'$ are \emph{equivalent} if they dominate each other. 

\begin{rem}\label{wordlengthdom}It is straightforward to check that if $G$ is finitely generated, any word length dominates any other length and hence two word lengths are always equivalent (as long as the generating sets are finite).\end{rem}
\begin{defi} We denote by $\cg$ the \emph{complex group ring} (or algebra) of the group $G$. It is a complex vector space with basis indexed by the elements of $G$, and the multiplicative structure is given by the group law. Analysts view $\cg$ as the set of functions $f:G\to{\C}$ with finite support, it is a ring for pointwise addition and convolution:
$$f*g(\gamma)=\sum_{\mu\in G}f(\mu)g(\mu^{-1}\gamma).\eqno{(f,g\in{\cg},\gamma\in G)}.$$
Algebraists tend to think of $\cg$ as formal linear complex combinations of elements in $G$, that is formal finite sums $\sum_{\gamma\in G}f_{\gamma}\gamma$ with the $f_{\gamma}$'s in $\C$, and the multiplicative structure is given by:
$$(\sum_{\gamma\in G}f_{\gamma}\gamma)(\sum_{\mu\in G}g_{\mu}\mu)=\sum_{\gamma,\mu\in G}f_{\gamma}g_{\mu}\gamma\mu.$$
Those two models are easily seen to be equal identifying the algebraists element $\gamma\in\cg$ with the function $\delta_{\gamma}:G\to\C$ taking value one in $\gamma$ and zero elsewhere and checking that $\delta_{\gamma}*\delta_{\mu}=\delta{\gamma\mu}$, so that a finitely supported function $f$ is just the finite combination $\sum_{\gamma\in G}f(\gamma)\delta_{\gamma}$. We shall without further warnings use both models, with a preference for the analyst's one.\end{defi}
It is a standard fact that any Hilbert space admits an orthonormal basis, so that any infinite dimensional Hilbert space with countable basis is isomorphic to $\ell^2({\N})$. We shall only consider separable Hilbert spaces, so in fact we shall only be dealing with ${\C}^n$'s or $\ell^2({\N})$. However, in the case of a discrete group $G$, writing $\ell^2(G)$ has the advantage of recording the unitary representation of $G$ on $\ell^2(G)$.
\begin{defi}\label{regrep}The \emph{(left) regular representation} of $G$ is the map
\begin{eqnarray*}G&\to&\UG\\
\gamma&\mapsto&\{\xi\mapsto\gamma(\xi):=\delta_\gamma*\xi\}
\end{eqnarray*}
This means that an element $\gamma\in G$ shifts $\xi\in\etg$ by pre-composition, since $\gamma(\xi)(g)=\delta_\gamma*\xi(g)=\xi(\gamma^{-1}g)$. Here $\UG$ denotes the unitary operators of $\etg$, and one can check that indeed, $\delta_\gamma^*=\delta_{\gamma^-1}$, so that $\delta_\gamma*\delta_\gamma^*=\delta_e$ which is the identity operator. The \emph{coefficients of the regular representation} are then given by
$$\{\left<\gamma(\xi),\eta\right>\}_{\gamma\in G},$$
where $\xi,\eta\in\etg$ are of norm one.
\end{defi}\label{conv}
Extending the regular representation by linearity induces an injective map
\begin{eqnarray*}\cg&\to&\BG\\
f&\mapsto&\sum_{\gamma\in G}f(\gamma)\delta_{\gamma}
\end{eqnarray*}
which is just the left convolution by $f$. Here $\BG$ denotes the algebra of bounded operators on $\etg$, the norm of such an operator being given by
$$\|f\|_*=\sup\{\|f*\xi\|_2\,|\,\xi\in\etg,\,\|\xi\|_2=1\}.$$
For $\xi\in\etg$ one has:
$$\|f*\xi\|_2=\|\sum_{\gamma\in G}f(\gamma)\delta_{\gamma}*\xi\|_2\leq\sum_{\gamma\in G}|f(\gamma)|\|\delta_{\gamma}*\xi\|_2=\sum_{\gamma\in G}|f(\gamma)|\|\xi\|_2=\|f\|_1\|\xi\|_2,$$
so that the supremum over the $\xi$'s of norm one will never exceed the right hand side of the above inequality and taking a limit in $\eog$ we hence deduce that 
$$\|f\|_*\leq\|f\|_1,$$ 
for all $f\in\eog$.
\begin{defi}\label{reducedCstar}
The \emph{reduced C*-algebra of $G$}, denoted by $\crg$ is the closure (for the operator norm) of $\cg$ acting on $\etg$ in its left regular representation, namely
$$\crg=\overline{\cg}^{\|\ \|_*}\subseteq\BG.$$
\end{defi}
\begin{rem}\label{prodscal} The inner product on $\etg$ is given by
$$\left<\xi,\eta\right>=\sum_{\gamma\in G}\xi(\gamma)\overline{\eta(\gamma)},$$
hence $f*\xi(e)=\sum_{\gamma\in G}f(\gamma)\xi(\gamma^{-1})=\left<f,\xi^*\right>$.
Taking $\xi=\delta_e$ one gets that $\|f\|_2\leq\|f\|_*\leq\|f\|_1$ showing that $\eog\subseteq\crg\subseteq\etg$. These are in general very rough estimates of $\|f\|_*$ unless $f$ is positive and $G$ is amenable and indeed Leptin's caracterization of amenability is the following.
\begin{thm}[Leptin \cite{Lep}]\label{Lep}Let $G$ be a group, then $G$ is amenable if and only of for any $f\in\R^+G$ then $\|f\|_*=\|f\|_1$.\end{thm}
The property of Rapid Decay will give a sharper estimate in case where the group is non-amenable. At that stage and using Definition \ref{RD1erjet} as well as Leptin's caracterization of amenability (Theorem \ref{Lep} above)  we can show the following.
\begin{thm}[Jolissaint \cite{Jolissaint}]\label{polgr}Let $G$ be a finitely generated amenable group. Then $G$ has the Rapid Decay property if and only if it has polynomial growth.\end{thm}
\begin{proof} Let $r\geq 0$ and denote by $\mathbb{1}_r$ the characteristic function of a ball of radius $r$. First assume that $G$ has polynomial growth (and hence is amenable). For $f$ supported on a ball of radius $r$ we compute
$$\|f\|_*\leq\|f\|_1=\left<\mathbb{1}_r,f\right>\leq\|\mathbb{1}_r\|_2\|f\|_2,$$
and conclude since $\|\mathbb{1}_r\|_2=\sqrt{\sharp B(r)}$ is a polynomial by assumption. The last inequality is just Cauchy-Schwarz inequality. 

Conversely assume that $G$ is amenable with the Rapid Decay property, hence
$$\|\mathbb{1}_r\|_*=\|\mathbb{1}_r\|_1\leq P(r)\|\mathbb{1}_r\|_2.$$
Since $\|\mathbb{1}_r\|_1=\sharp B(r)=\|\mathbb{1}_r\|_2^2$ we conclude that $\sqrt{\sharp B(r)}\leq P(r)$ and hence $G$ has polynomial growth.
\end{proof}
\end{rem}
\begin{defi}\label{rdfunctions} Consider for $f$ in $\cg$ and $\ell$ a length function on $G$ a {\it weighted $\ell^2$ norm}, depending on a parameter $s\geq 0$ and given by
$$\|f\|_{\ell,s}=\sqrt{\sum_{\gamma\in G}|f(\gamma)|^2(1+\ell(\gamma))^{2s}}.$$ 
The \emph{$s$-Sobolev completion} of $\cg$ is the completion of $\cg$ with respect to this norm and is denoted by $H^s_\ell(G)$. The \emph{functions of rapid decay} are given by
$$H^\infty_\ell(G)=\bigcap_{s\geq 0}H^s_\ell(G)$$
\end{defi}
If $s=0$ or $\ell=0$, then $H^s_\ell(G)=\etg$. The $H^s_\ell(G)$'s are Hilbert spaces for the scalar product 
$$\left<f,g\right>_{\ell}=\sum_{\gamma\in G}f(\gamma)\overline{g(\gamma)}(1+\ell(\gamma))^{2s},$$ 
$H^\infty_\ell(G)$ is a Fr\'echet space, but a priori none of them are algebras unless $G$ has the Rapid Decay property, in which case, Lafforgue's adaptation of a result of Jolissaint \cite{JolK} gives the following.
\begin{prop}[Lafforgue \cite{Lafforgue}]\label{HsBanach}Let $G$ be a group with the Rapid Decay property for a length function $\ell$.  Then there is $s_0$ big enough so that for any $s\geq s_0$, the Hilbert space $H^s_\ell(G)$ is in fact a Banach sub-algebra of $\crg$, which is dense and has the same K-theory as $\crg$.\end{prop}
An alternate definition of the Rapid Decay property, giving a sharper estimate for the operator norm, is the following.
\begin{defi}Let $G$ be a finitely generated group, the $G$ has \emph{the Rapid Decay property} if there are constants $c\geq 1$ and $s\geq 0$ such that, for every $f\in\crg$
$$\|f\|_*\leq c\|f\|_{\ell,s}\eqno{RD(3)}$$
for one (hence any) algebraic length $\ell$ on $G$.
\end{defi}
\begin{rem}The Rapid Decay property depends on the choice of a length on $G$. However, because of Remark \ref{wordlengthdom} it is straightforward to see that if $G$ finitely generated has the Rapid Decay property with respect to any length, then it will have it for the word length as it dominates all the others. Moreover, the group $G$ will have the Rapid Decay property for one length if and only if it has it for any equivalent length.\end{rem}
\begin{rem}\label{Johnny}Taking the complex group algebra is a functor from the category of groups to the one of complex algebras, and it can be viewed both as a covariant or contravariant functor. In fact, algebraists tend to think of a covariant functor, meaning that given a group homomorphism $\varphi:G\to\Lambda$ one gets an algebra homomorphism $\varphi_*:\cg\to\C\Lambda$ by $\varphi_*(\sum_{\gamma\in G}f_{\gamma}\gamma)=\sum_{\gamma\in G}f_{\gamma}\varphi(\gamma)$, whereas analysts see a contravariant functor with $\varphi^*:\C\Lambda\to\cg$ by $\varphi^*(f)=f\circ\varphi$. The same holds for $\ell^1$, but not the reduced C*-algebra, and this is an important problem in the Baum-Connes conjecture, as the rest of the objects defining it are all functorial.\end{rem}
\section{The case $G=\Z$}\label{Z}
Let us look at the Rapid Decay property in the case where the group is $\Z$.  Using inequality RD(1) it is straightforward to show that $\Z$ has the Rapid Decay property, as for $f\in\cg$ supported on elements of length less than $R$, since the operator norm is always bounded above by the $\ell^1$ norm, we have that
$$\|f\|_*\leq\|f\|_1=\sum_{n=-R}^R|f(n)|\leq\sqrt{\sum_{n=-R}^R1}\sqrt{\sum_{n=-R}^R|f(n)|^2}=\sqrt{2}R^{1/2}\|f\|_2$$
The last inequality is Cauch-Schwartz inequality, and the same computation in fact establishes the Rapid Decay property for groups with polynomial volume growth, and the exponent D being half of the growth, see Theorem \ref{polgr}.

Let us now look at Connes' original definition for the case $G=\Z$. To do so, first let $S^1=\{z\in\C||z|=1\}$ denote the circle and $\LS$ denote the Laurent polynomials over the circle, that is, finite sums of the form
$$A=\sum_{n\in\Z}A_nz^n.$$
We denote by $\LDS$ the square integrable functions (for the Lebesgue measure on $S^1$), by $\CS$ the continuous functions and by $\DS$ the smooth functions. We shall see that the Fourier transform is an algebra isomorphism $\C\Z\simeq\LS$, that extends to an isometric isomorphism of Hilbert spaces $\ell^2(\Z)\simeq\LDS$, and an isometric isomorphism of $C^*$-algebras $C^*_r(\Z)\simeq\CS$ which restricted to $H^{\infty}(\Z)$ has image $\DS$.

Let us start by recalling that $\LS$ is an algebra for pointwise addition and multiplication, and notice that $\LS$ is a prehilbert space, as it is a complex vector space and
\begin{eqnarray*}\left<\ ,\ \right>:\LS\times\LS&\to &\C\\
(A,B)&\mapsto &\int_{S^1}A(z)\overline{B(z)}dz\end{eqnarray*}
defines an inner product, that by definition extends to the Hilbert space
$$L^2(S^1)=\{F:S^1\to\C|\int_{S^1}|F(z)|^2dz\leq\infty\}$$
 Moreover, $\{z^n\}_{z\in\Z}$ is an orthonormal basis for $\LS$ since
$$\int_{S^1}z^ndz=\left\{\begin{array}{cc}0 & \hbox{ if }n\ne 0\\
1 & \hbox{ if }n=0\end{array}\right.$$
and a Hilbert basis for $L^2(S^1)$.
\begin{defi}Let $G$ be a finitely generated abelian group, its \emph{Pontryagin dual} is the compact abelian group
$$\hat{G}={\rm Hom}(G,S^1).$$
The \emph{Fourier transform} is the map
\begin{eqnarray*}\ch:\cg&\to &{\mathcal C}(\hat{G})\\
f&\mapsto & \ch(f):\{x\mapsto\sum_{\gamma\in G}f(\gamma)x(\gamma)\}.\end{eqnarray*}\end{defi}
In case where $G=\Z$, then $\hat{G}=S^1$ since a homomorphism is then determined by its value on 1, and the Fourier transform reads
\begin{eqnarray*}\ch:\C\Z&\to &\CS\\
f&\mapsto & \ch(f)(z)=\sum_{n\in\Z}f(n)z^n,\end{eqnarray*}
which is an isomorphism from $\C\Z$ to $\LS$ since it is a bijection on the basis. Moreover, since
$$\ch(\delta_n*\delta_m)=\ch(\delta_{n+m})=z^{n+m}=z^nz^m=\ch(\delta_n)\ch(\delta_m)$$
we conclude that $\ch$ is an algebra isomorphism. We have hence proved the following.
\begin{prop}The algebras $\C\Z$ and $\LS$ are isomorphic.\end{prop}
Let us now turn to $\ell^2(\Z)$ and $\LDS$. It can seem obvious now that the Fourier transform $\ch$ extends to an isometric isomorphism between $\ell^2(\Z)$ and $\LDS$, but there are some subtle points worth being mentionned. Indeed if $\C\Z$ is clearly dense in $\ell^2(\Z)$, it is less clear why $\LS$ should be dense in $\LDS$. It is hence starightforward to check that the Fourier transform $\ch$ extends to an isometric injective linear map $\ch:\ell^2(\Z)\to\LDS$, but surjectivity is a bit subtle. Indeed, to show that $\ch$ is surjective, it is tempting to pick $F\in\LDS$ and try to express it as a series $\sum_{n\in\Z}a_nz^n$, where $a$ such that $a(n)=a_n$ belongs to $\ell^2(\Z)$, we even know that the $a_n$'s should be the \emph{Fourier coefficients} of $F$, given by
$$a_n=\int_{S^1}F(z)z^{-n}dz,$$
since the map ${\mathcal I}:\LS\to\C\Z$, $A\mapsto{\mathcal I}(A)(n)=\int_{S^1}A(z)z^{-n}dz$ is an inverse for $\ch$ on $\C\Z$. As such, this doesn't quite work (but almost) because there are continuous functions whose Fourier series diverge at some points (those examples have been provided by Du Bois-Reymond in 1876, see e.g. \cite{Bhatia} p. 50).

However if we admit the very well-known fact that the space of smooth functions $\DS$ is dense in $\LDS$, we can use the following particular case of Dirichlet's Theorem (see \cite{Bhatia} for the general case) to show that $\LS$ is dense in $\DS$ and hence in $L^2(S^1)$ as well.
\begin{thm}The Fourier series of an $F\in\DS$ uniformly converges to $F$ at every point.\end{thm}
Let us now turn to the $C^*$-algebras and recall that $\CS$ is a $C^*$-algebra for pointwise addition and multiplication and for the sup norm, given by
$$\|F\|=\sup\{|F(z)|,z\in S^1\}\ \ (F\in\CS).$$
It is a standard application of the Stone-Weierstra\ss\ theorem that the space $\LS$ is dense in $\CS$ for the supremum norm (but notice that in view of Dirichlet's theorem and the examples of continuous functions having Fourier expansions divergent in some point, we know that the polynomials approching pointwise a continuous function is not necessarily a partial sum of its Fourier expansion).
\begin{lemma}Take $F\in\CS$ and consider the operator
\begin{eqnarray*}L_F:\LDS&\to &\LDS\\
\xi&\mapsto &F\xi.\end{eqnarray*}
then its operator norm $\|L_F\|_*$ is equal to $\|F\|$.
\end{lemma}
\begin{proof}Let $z_0\in S^1$ be such that $|F(z_0)|=\|F\|$. By definition, $\|L_F\|_*=\sup\{\|F\xi\|_2|\xi\in\LDS,\|\xi\|_2=1\}$. Moreover, for $\xi\in\LDS$ we have that
$$\|F\xi\|_2=\sqrt{\int_{S^1}|F(z)\xi(z)|^2dz}\leq\|F\|\sqrt{\int_{S^1}|\xi(z)|^2dz}=\|F\|\|\xi\|_2,$$
so that $\|L_F\|_*\leq\|F\|$. To approach $\|F\|$ with elements of the form $\|F\xi\|_2$ and $\|\xi\|_2=1$, take $\epsilon>0$ and let $U_{\epsilon}$ be a neighborhood of $z_0$ such that $|F(z)|>|F(z_0)|-\epsilon>0$ for any $z\in U_{\epsilon}$ (if $\|F\|=0$ it's obvious). Define
$$\xi_{\epsilon}(z)=\left\{\begin{array}{cc}|U_{\epsilon}|^{-1/2}&\hbox{ if }z\in U_{\epsilon}\\
0&\hbox{ otherwise}\end{array}\right.$$
Then $\|\xi_{\epsilon}\|_2=1$ and $\|F\xi_{\epsilon}\|_2\geq |F(z_0)|-\epsilon=\|F\|-\epsilon$.
\end{proof}
It is now straightforward to deduce the following.
\begin{prop}The algebra isomorphism $\ch:\C\Z\to\LS$ extends to an isometric isomorphism $C^*_r(\Z)\to\CS$.\end{prop}
We now end this discussion by looking at smooth functions. Notice that, for $F\in{\mathcal C}^1(S^1)$, if $F(z)=\sum_{n\in\Z}A_nz^n$, its derivative is given by
$$F'(z)=\sum_{n\in\Z}inA_nz^n,$$ 
so that in case where $F\in{\mathcal C}^k(S^1)$, the $k$-th derivative is given by 
$$F^{(k)}(z)=\sum_{n\in\Z}i^kn^kA_nz^n.$$ 
Since everything lives in $\LDS$, we have that $F\in\CS$ is actually in $\DS$ if and only if all its derivatives are in $\LDS$, which means that
$$\sum_{n\in\Z}|n|^{2k}|A_n|^2<\infty$$
for any $k\in\N$. With a minor change of $|n|$ by $(1+\ell_{\Z}(n))$ (up to a constant, depending on the chosen finite generating set for $\Z$) we see that $F\in\DS$ if and only if it is the Fourier transform of an element in $H^{\infty}(\Z)$. To summarize, we have obtained the following identifications:
\begin{eqnarray*}\Z &\longleftrightarrow & S^1\\
\C\Z & \simeq &\LS\\
H^{\infty}(\Z)&\simeq &\DS\\
C^*_r(\Z)&\simeq&\CS\\
\ell^2(\Z)&\simeq &\LDS
\end{eqnarray*}
In this case, the Rapid Decay property for $\Z$ in the sense of the inclusion RD(2) amounts to the well-known inclusion $\DS\subseteq\CS$ because smooth functions are in particular continuous.
\section{Equivalent definitions of property RD}\label{equivalences}
In this section we shall see several equivalent definitions of the Rapid Decay property, giving us a bigger flexibility for using that property. See also \cite{Bo} for more equivalent definitions.
\begin{prop}\label{RD}Let $G$ be a discrete group, endowed with a length function $\ell$. Then the following are equivalent:
\begin{enumerate}
\item\label{RD1} There are constants $C$ and $D$ such that for any $R\in\N$ and any $f\in\cg$ such that $f$ is supported on elements shorter than $R$, the following inequality holds:
$$\|f\|_*\leq CR^D\|f\|_2,\eqno{RD(1)}$$
\item\label{RD2} The following containment holds
$$H^\infty_\ell(G)\subseteq C^*_r(G)\eqno{RD(2)}$$
\item\label{RD3}There exists constants $c,s\geq 0$ such that, for each $f\in\cg$ one has
$$\|f\|_{*}\leq c\|f\|_{\ell,s}\eqno{RD(3)}$$
\item \label{RD4} There are constants $C$ and $D$ such that for any $R\in\N$ and any $f,g\in\cg$ such that $f$ is supported on elements shorter than $R$, the following inequality holds:
$$\|f*g\|_*\leq CR^D\|f\|_2\|g\|_2,\eqno{RD(4)}$$
\item\label{RD5}There are constants $C$ and $D$ such that for any $R\in\N$ and any $f,g,h\in\cg$ such that $f$ is supported on elements shorter than $R$, the following inequality holds:
$$|f*g*h(e)|\leq CR^D\|f\|_2\|g\|_2\|h\|_2\eqno{RD(5)}$$
\item\label{RD6} There are constants $M$ and $k\geq 0$ such that for any $\xi,\eta\in\etg$ or norm one, the coefficients of the regular representation satisfy
$$\sum_{\gamma\in G}{{|\left<\gamma(\xi),\eta\right>|^2}\over{(1+\ell(\gamma))^{2k}}}\leq M\eqno{RD(6)}$$
\item\label{RD7} There are constants $C$ and $D$ such that for any $R\in\N$ and for any $\xi,\eta\in\etg$ or norm one, we have that the coefficients of the regular representation satisfy
$$\sum_{\gamma\in B_R}|\left<\gamma(\xi),\eta\right>|^2\leq CR^{2D}.$$
\end{enumerate}\end{prop}
Condition RD(6) was first given by Breuillard at an AIM conference on the property of Rapid Decay, and RD(7) is due to Perrone in \cite{Perrone}.
\begin{rem}\label{Rplus}There are variations on each equivalent conditions:
\begin{enumerate}
\item For instance one can always choose $f\in{\R}_+G$ instead of $\cg$. Indeed, to go from ${\bf R}_+G$ to $\cg$ (for instance in RD(1), we write $f\in\cg$ as $f=f_1-f_2+i(f_3-f_4)$ with $f_i\in{\bf R}_+G$ and the supports of $f_i$ and $f_{i+1}$ disjoints for $i=1,3$, then $\|f\|^2_2=\sum_{i=1}^4\|f_i\|^2_2$ and thus
$$\|f\|_*\leq\sum_{i=1}^4\|f_i\|_*\leq CR^D\sum_{i=1}^4\|f_i\|_2\leq \sqrt{4}CR^D\|f\|_2.$$
\item All those conditions are expressed in terms of the regular representation (see Definition \ref{regrep}) but most conditions can be reformulated in terms of any unitary representation with some obvious changes. 

\item The length does not need to be a word length, a similar string of equivalences holds for other length (even in case where $G$ is not finitely generated), and the proofs are easily adapted (the length may not have integer values only).
\end{enumerate}
The multiplicative constants are unimportant and vary from one point to another. The exponent are more interesting and do vary as well. It is the same exponent $D$ for points (1), (4), (5) and (7), but the exact relation between $D,s$ and $k$ seems unclear. According to Nica in \cite{Nica}, the degree in the sense of RD(3) has to be greater than 1/2 and knowing the exact degree has interesting consequences in view of Theorem \ref{CPS}, but most methods for establishing the Rapid Decay property give a bad estimate for that degree.\end{rem}
\begin{proof} We will show that $(1)\Longleftrightarrow (3)$, then $(2)\Longleftrightarrow (3)$, and $(1)\Longrightarrow(4)\Longrightarrow(5)\Longrightarrow(1)$, then $(3)\Longleftrightarrow (6)$, and finally $(1)\Longleftrightarrow (7)$.

\medbreak

\underline{$(3) \Longleftrightarrow (1)$}: Take $f\in\cg$ with support contained in $B_R$ (a ball of radius $R$), we have assuming (3) that:
\begin{eqnarray*}\|f\|_*\leq c\|f\|_{\ell,s}=c\sqrt{\sum_{\gamma\in B_R}|f(\gamma)|^2(R+1)^{2s}}\leq c(R+1)^s\|f\|_2\leq CR^s\|f\|_2\end{eqnarray*}
and thus (1) is satisfied, for $D=s$ and $C$ depending on $c$ and $s$. 

Conversely, we denote, for $n\in\N$ by $S_n=\{\gamma\in G|\ell(\gamma)=n\}$ the sphere of radius $n$ and compute, for $f\in\cg$:
$$\|f\|_*=\|\sum_{n=0}^\infty f|_{S_n}\|_*\leq\sum_{n=0}^\infty\|f|_{S_n}\|_*.$$
So, assuming (1) we have that 
\begin{eqnarray*}\|f\|_*&\leq &\sqrt{|f(e)|}+\sum_{n=1}^\infty Cn^D\|f|_{S_n}\|_2\leq \sum_{n=0}^\infty K(n+1)^D\|f|_{S_n}\|_2\\
&\leq&K\sum_{n=0}^\infty (n+1)^{-1}(n+1)^{D+1}\|f|_{S_n}\|_2\\
&\leq&K\underbrace{\sqrt{\sum_{n=0}^\infty (n+1)^{-2}}}_{\pi/\sqrt{6}}\sqrt{\sum_{n=0}^\infty (n+1)^{2D+2}\|f|_{S_n}\|_2^2}\\
&\leq &c\sqrt{\sum_{n=0}^\infty\sum_{\gamma\in S_n}|f(\gamma)|^2(\ell(\gamma)+1)^{2D+2}}=c\|f\|_{\ell,D+1}.\end{eqnarray*}
Hence we get (3) for $s=D+1$. 

\medbreak

\underline{$(2) \Longleftrightarrow  (3)$}: First notice that RD(3) is equivalent to $H^s_\ell(G)\subseteq \crg$, so that $(3)$ implies $(2)$ is obvious since $H^\infty_\ell(G)\subseteq H^s_\ell(G)$, hence we will be looking at the other implication. Let us first prove that the graph of the inclusion $H^\infty_\ell(G)\to\crg$ is closed. Indeed, let $\{f_n\}_{n\in\N}$ in $H^\infty_\ell(G)$ tend to $f$ in in $H^\infty_\ell(G)$ and the image under the inclusion $\{f_n \}_{n\in\N}$ tend to $g$ in $\crg$, we need to prove that $f=g$. Since operator convergence implies weak *-convergence, $\{\left<f_n*\xi,\eta\right>\}_{n\in\N}$ converges to $\left<g*\xi, \eta\right>$ for any $\xi,\eta\in\etg$. Since $\{f_n\}_{n\in\N}$ tends to $f$ in $\etg$ as well, it implies that $\{f_n*\xi*\eta(e)\}_{n\in\N}$ converges to $f *\xi*\eta(e)$ for any $\xi,\eta\in\cg$. According to Remark \ref{prodscal}, $f *\phi*\psi(e)=\left<f*\phi, \psi^*\right>$, so we conclude that $f = g$. The closed graph theorem in the generality of Proposition 1, Chapter I page 20 of \cite{Bou}, applied to the Fr\'echet space $H^\infty_\ell(G)$ and to $\crg$ (viewed as a Banach space) then implies that the inclusion RD(2) is continuous. This by definition amounts to the existence of an $s>0$ and $c>1$ such that $\|f\|_*\leq c\|f\|_{\ell,s}$ for any $f\in H^\infty_\ell(G)$, an hence in particular for any $f\in\cg$.

\medbreak

\underline{$(1)\Longrightarrow(4)\Longrightarrow(5)\Longrightarrow(1)$}: That (1) implies (4) is obvious since for any $g\in\cg$, by definition of the operator norm of $f$ we have that 
$${\|f*g\|_2\over{\|g\|_2}}\leq\|f\|_*.$$
That (4) implies (5) follows from Cauchy-Schwartz inequality:
$$|f*g*h(e)|\leq\sum_{\gamma\in G}|f*g(\gamma)h^*(\gamma)|\leq\|f*g\|_2\|h\|_2$$
where $h^*(\gamma)=\overline{h(\gamma^{-1})}$.
To see that (5) implies (1) it is enough to define, for $\gamma\in G$
$$h(\gamma)=\frac{f*g(\gamma^{-1})}{\|f*g\|_2},$$
and notice that in that case $f*g*h(e)=\|f*g\|_2$ and $\|h\|_2=1$. Then, for any $\epsilon>0$, by the definition of the operator norm and by density of $\cg$ in $\etg$, there is $g_\epsilon\in\cg$ such that:
$$\|f\|_*-\epsilon\leq\frac{\|f*g_\epsilon\|_2}{\|g_\epsilon\|_2}$$
and since the above inequality holds for any $\epsilon>0$ we recover point (1).

\medbreak

\underline{$(3) \Longleftrightarrow (6)$}: Let us first assume RD(3) and take $k=s+1$. For any $n\in\N$ we denote by $S_n$ the sphere of radius $n$, that is the elements in $G$ of length exactly equal to $n$. We set
$$a_n(\gamma)=\mathbb{1}_{S_n}(\gamma){\overline{\left<\gamma(\xi),\eta\right>}\over{(1+n)^{2k}}}$$
and compute
\begin{eqnarray*}\sum_{\gamma\in S_n}{|\left<\gamma(\xi),\eta\right>|^2\over{(1+\ell(\gamma))^{2k}}}&=&\sum_{\gamma\in S_n}\left<\gamma(\xi),\eta\right>a_n(\gamma)=\left<a_n(\xi),\eta\right>\leq \|a_n\|_*\\
&\leq & c\|a_n\|_{\ell,s}=c\sqrt{\sum_{\gamma\in S_n}{|\left<\gamma(\xi),\eta\right>|^2\over{(1+\ell(\gamma))^{4k}}}(1+\ell(\gamma))^{2s}}\\
&=&c{1\over (1+n)}\sqrt{\sum_{\gamma\in S_n}{|\left<\gamma(\xi),\eta\right>|^2\over{(1+\ell(\gamma))^{2k}}}}.\end{eqnarray*}
Where $c$ is a constant that does not depend on $\xi$ and $\eta$. Hence,
$$\sum_{\gamma\in G}{{|\left<\gamma(\xi),\eta\right>|^2}\over{(1+\ell(\gamma))^{2k}}}=\sum_{n=1}^\infty\sum_{\gamma\in S_n}{|\left<\gamma(\xi),\eta\right>|^2\over{(1+\ell(\gamma))^{2k}}}\leq \sum_{n=1}^\infty c{1\over (1+n)^2}\leq M.$$
Conversely, assuming RD(6), we take $f\in\cg$ supported on a ball of radius $R$ and $\xi\in\etg$ of norm one and so that $f(\xi)\not=0$ (for instance any Dirac mass would do), then, for $\eta=f(\xi)/\|f(\xi)\|_2$ we have that (using Cauchy-Schwartz inequality):
\begin{eqnarray*}\|f(\xi)\|_2&=&\left<f(\xi),\eta\right>=\sum_{\gamma\in G}|f(\gamma)|\left<\gamma(\xi),\eta\right>=\sum_{\gamma\in G}{\left<\gamma(\xi),\eta\right>\over{(1+\ell(\gamma))^k}}|f(\gamma)|(1+\ell(\gamma))^k\\
&\leq &\sqrt{\sum_{\gamma\in G}{\left<\gamma(\xi),\eta\right>\over{(1+\ell(\gamma))^k}}}\|f\|_{\ell,k}\leq M\|f\|_{\ell,k}.\end{eqnarray*}
Since the operator norm of $f$ is the supremum over all $\xi\in\etg$ of norm one of the left hand side of the above inequality and $M$ does not depend on $\xi$, we deduce RD(3), with $s=k$.

\medbreak

\underline{$(1) \Longleftrightarrow (7)$}: To prouve the direct implication, take $\xi,\eta\in\etg$ of norm one, $R\geq 1$ and define
$$f(\gamma)=\mathbb{1}_{B_R}\overline{\left<\gamma(\xi),\eta\right>},$$
we compute
\begin{eqnarray*}\sum_{\gamma\in B_R}|\left<\gamma(\xi),\eta\right>|^2&=&\sum_{\gamma\in B_R}\left<\gamma(\xi),\eta\right>f(\gamma)=\left<f(\xi),\eta\right>\leq\|f\|_*\leq CR^D\|f\|_2\\
&=&CR^D\sqrt{\sum_{\gamma\in B_R}|\left<\gamma(\xi),\eta\right>|^2},\end{eqnarray*}
and RD(7) follows. Conversely assuming RD(7), we take $f\in\cg$ and $\xi\in\etg$ of norm one and so that $f(\xi)\not=0$ (for instance any Dirac mass would do), then, for $\eta=f(\xi)/\|f(\xi)\|_2$ we have that (using Cauchy-Schwartz inequality):
\begin{eqnarray*}\|f(\xi)\|_2&=&\left<f(\xi),\eta\right>=\sum_{\gamma\in G}f(\gamma)\left<\gamma(\xi),\eta\right>=\sum_{\gamma\in B_R}f(\gamma)\left<\gamma(\xi),\eta\right>\\
&\leq &\|f\|_2\sqrt{\sum_{\gamma\in B_R}|\left<\gamma(\xi),\eta\right>|^2}\leq CR^D\|f\|_2.\end{eqnarray*}
Since the operator norm of $f$ is the supremum over all $\xi\in\etg$ of norm one of the left hand side of the above inequality and $C$ does not depend on $\xi$, we deduce RD(1), with the same exponent $D$.
\end{proof}
\begin{rem}\label{otherlength}\label{sousgroupes} \begin{enumerate}\item Using either RD(\ref{RD1}) or RD(\ref{RD3}), we immediately see that any subgroup $H$ of $G$ has property RD for the induced length. Indeed, since if $H$ is a subgroup of $G$, $f\in\C H$ supported in a ball of radius $r$ can be viewed in $\cg$, supported in a ball of radius $r$ as well and the right hand side doesn't change, whereas on the left hand side, the operator norm of $f$ acting on $\ell^2(H)$ is smaller than the operator norm of $f$ acting on $\etg$ since there are more elements in $\etg$ to evaluate $f$ on.
\item According to Remark \ref{wordlengthdom}, the world length dominates any other length, so balls in the word length are the smallest, hence if a group has the Rapid Decay property with respect to any thength, then in particular it will have the Rapid decay property with respect to any word length. This is the reason why we often omit specifying the length, assuming that we only deal with the word length.
\end{enumerate}
\end{rem}
\begin{ex} For a discrete group $G$, the map $\ell_0:G\to{\R}_+$ defined by $\ell_0(1)=0$ and $\ell_0(\gamma)=K$ for any $\gamma\in G$ and some constant $K$ is a length function, and $G$ has the Rapid Decay property with respect to $\ell_0$ if and only if $G$ is finite. Indeed, if $G$ has the Rapid Decay property with respect to $\ell_0$, then there exists a constant $C$ such that for any $f,g\in{\cg}$ then $\|f*g\|_2\leq C\|f\|_2\|g\|_2$, which implies that $\etg$ is an algebra. This can happen if and only if $G$ is finite, see \cite{Raj}. The same statement holds if we just assume $\ell_0$ to be bounded.\end{ex}


\end{document}